\documentclass{amsart}
\usepackage{amssymb}
\usepackage{graphicx}

\DeclareSymbolFont{AMSb}{U}{msb}{m}{n}
\DeclareSymbolFontAlphabet{\Bb}{AMSb}

\usepackage{bbm}

\usepackage{tikz}

\usetikzlibrary{intersections,calc,positioning}
\usetikzlibrary{shapes.geometric,fit}
\usetikzlibrary{decorations,arrows,decorations.pathmorphing,backgrounds,positioning,fit,petri,shadows}
\usetikzlibrary{3d}
\usepackage{yhmath}
\usepackage{mathdots}
\usepackage{MnSymbol}

\newtheorem{theorem}{Theorem}[section]
\newtheorem{lemma}[theorem]{Lemma}

\theoremstyle{remark}

\newtheorem{remark}[theorem]{Remark}

\numberwithin{equation}{section}
\begin{document}

\title[Combinatorial relations among relations]
{Combinatorial relations among relations for level 2 standard $C_{n}\sp{(1)}$-modules}
\author{Mirko Primc and Tomislav \v Siki\' c }
\address{Mirko Primc, University of Zagreb, Faculty of Science, Bijeni\v cka 30, 10000 Zagreb, \mbox{\hskip 7.6em} Croatia}
\email{primc@math.hr}
\address{Tomislav \v{S}iki\'{c}, University of Zagreb, Faculty of Electrical Engineering and Com- \mbox{\hskip 8.5em} puting, Unska 3, 10000 Zagreb,  Croatia}
\email{tomislav.sikic@fer.hr}


\begin{abstract}
For an affine Lie algebra $\hat{\mathfrak g}$ the coefficients of certain vertex operators which annihilate level $k$ standard $\hat{\mathfrak g}$-modules are the defining relations for level $k$ standard modules. In this paper we study a combinatorial structure of the leading terms of these relations for level $k=2$ standard  $\hat{\mathfrak g}$-modules for affine Lie algebras of type $C_{n}\sp{(1)}$ and the main result is a construction of combinatorially parameterized relations among the coefficients of annihilating fields. It is believed that the constructed relations among relations will play a key role in a construction of Groebner-like basis of the maximal ideal of the universal vertex operator algebra $V^ k_{\mathfrak g}$ for $k=2$.
\end{abstract}
\maketitle
\def\sq{{\lower.3ex\vbox{\hrule\hbox{\vrule height1.2ex depth1.2ex\kern2.4ex			\vrule}\hrule}\,}}
\section{Introduction}

Let ${\mathfrak g}$ be a simple complex Lie algebra and $\hat{\mathfrak g}$ the corresponding affine Kac-Moody Lie algebra (cf. \cite{K}). In \cite{MP1} combinatorial bases of all standard (i.e. integrable highest weight) $\hat{\mathfrak g}$-modules $L(\Lambda)$ were constructed for ${\mathfrak g}={\mathfrak sl}_2$. A part of that construction can be carried through for the vacuum $\hat{\mathfrak g}$-modules $L(k\Lambda_0)$ for all simple ${\mathfrak g}$ and all positive integer $k$. The key ingredient is the vertex operator algebra structure\footnote{
$N(k\Lambda_0)$ is often denoted as $V^k({\mathfrak g})$ or $V_{\hat{\mathfrak g}}(k,0)$ (cf. \cite{LL}), but we shall use the notation from \cite{PS1} and \cite{PS2}.
} 
on the generalized Verma $\hat{\mathfrak g}$-module $N(k\Lambda_0)$ generated by fields $$
x(z)=\sum_{m\in\mathbb Z}x(m)z^{-m-1}=Y(x(-1){\bf 1}, z), \quad x\in\mathfrak g,
$$
where $x(m)=x_m=x\otimes t^m$ is the usual notation for the elements of the affine Lie algebra and $Y(v,z)$ is the usual notation for vertex operators. We also write \newline $\mathfrak g=\mathfrak g\otimes t^0\subset\hat{\mathfrak g}$ and $\hat{\mathfrak g}_{<0}=\coprod_{m<0}{\mathfrak g}\otimes t^m$.
\smallskip

Let $R$ be the finite dimensional ${\mathfrak g}$-module  generated by the singular vector in $N(k\Lambda_0)$, i.e.
$$
R=U({\mathfrak g})\cdot x_\theta(-1)^{k+1}{\bf 1} \cong L_{\mathfrak g}((k+1)\theta),
$$
where $x_\theta$ is a root vector for the maximal root $\theta$ with respect to a chosen Cartan decomposition of ${\mathfrak g}$. Then the coefficients $r(m)$, $r\in R$, $m\in\mathbb Z$, of vertex operators
$$
Y(r,z)=\sum_{m\in\mathbb Z}r(m)z^{-m-k-1}
$$
span a loop $\hat{\mathfrak g}$-module $\bar R$. Since  $\bar RN(k\Lambda_0)\subset N(k\Lambda_0)$ is the maximal submodule of the generalized Verma module (cf. \cite{MP1}, \cite{PS1}, \cite{PS2} and the references therein) we have
\begin{equation}\label{E:relations for L(Lambda)}
	L(k\Lambda_0)=N(k\Lambda_0)/\bar RN(k\Lambda_0)
	\quad\text{and}\quad \bar R\,|_{L(k\Lambda_0)}=0,
\end{equation}
and for this reason we call the elements of $\bar R$ {\it the relations} for level $k$ standard (vacuum) $\hat{\mathfrak g}$-modules. Note that (\ref{E:relations for L(Lambda)}) is the consequence of representation theory of affine Kac-Moody Lie algebras and the associated vertex operator algebras. 

One can use the relations $\bar R$ to construct a combinatorial bases of $L(k\Lambda_0)$---the basic idea is to reduce the PBW spanning set of $L(k\Lambda_0)$ to a basis $\mathcal B$ by using relations $r|_{L(k\Lambda_0)}=0$, and to parameterize the monomial vectors 
$$
u(\pi){\bf 1}\in\mathcal B\subset L(k\Lambda_0)=U(\hat{\mathfrak g}){\bf 1}
$$ 
with monomials $\pi$ in the symmetric algebra $S(\hat{\mathfrak g})$. This is done in steps:
\begin{itemize}
	\item Choose an (appropriately) totally ordered weight basis $B$ of $\mathfrak g$ and the corresponding basis $\hat B=\bar B\cup\{c\}$ of $\hat{\mathfrak g}$, where $c$ is the canonical central element in $\hat{\mathfrak g}$ and 
	$$
	\bar B=\{b(m)\mid b\in B, m\in\mathbb Z\}.
	$$ 
	We extend the strict order $\prec$ on $B$ to $\bar B$ so that $m< m'$ implies $b(m)\prec  b'(m')$.
	Since $b(m){\bf 1}=0$ for $m\geq0$, in some arguments it is enough to consider basis elements in
	$$
	\bar B_{<0}=\{b(m)\mid b\in B, m<0\}=\bar B\cap \hat{\mathfrak g}_{<0}.
	$$
	\item Denote by $\mathcal P$ the set of monomials
	\begin{equation*}
		\pi=\prod_{b(j)\in \bar B}b(j)^{n_{b(j)}}\in S(\hat{\mathfrak g})
	\end{equation*}
	and by $\mathcal P_{<0}= \mathcal P\cap S(\hat{\mathfrak g}_{<0})$. Sometimes we say that  $\pi\in\mathcal P_{<0}$ is a colored partition of length $\ell(\pi)$,  degree $|\pi|$ 
	 and support $\text{supp\,}\pi$,
	$$
	\quad\qquad\ell(\pi)=\sum_{b(j)\in \bar B}{n_{b(j)}}, \ \ 
	|\pi|=\sum_{b(j)\in \bar B}{n_{b(j)}}\cdot j, \ \ 
	\text{supp\,}\pi =\{b(j)\in\bar B_{<0}\mid n_{b(j)}>0\},
	$$
	with colored parts ${b(j)}\in\text{supp\,}\pi$ of degree $j<0$ and color $b\in B$, appearing in the partition $n_{b(j)}$ times. For $\ell\geq0$ and $m\in\mathbb Z$ set
	$$
	\mathcal P^\ell =\{\pi\in\mathcal P\mid \ell(\pi)=\ell\},\quad
	\mathcal P^\ell(m) =\{\pi\in\mathcal P\mid \ell(\pi)=\ell, |\pi|=m\}.
	$$
	Instead of the product $\mu\nu\in S(\hat{\mathfrak g}_{<0})$ we shall write $\mu\cup\nu\in {\mathcal P}_{<0}$ saying that $\mu\cup\nu$ is the partition having all the parts of $\mu$ and $\nu$ together. Likewise we write
	$$
	\rho\subset\pi\quad\text{if}\quad \pi=\rho\kappa
	$$
	for some $\kappa$, saying that $\rho\subset\pi$ is an {\it embedding of $\rho$ into $\pi$}. We extend the total order $\prec$ on $\bar B$ to $\mathcal P$ so that $\mu\prec  \mu'$ implies $\mu\kappa\prec  \mu'\kappa$. For arguments by induction we also need that  $\prec $ on  $\mathcal P_{<0}$ is a well order and that $\ell(\pi)>\ell(\pi')$ or $|\pi|<|\pi'|$ implies $\pi\prec \pi'$.
	\smallskip
	
	\item For the filtration $U_\ell(\hat{\mathfrak g})$, $\ell\geq0$, let
	$$
	p_\ell\colon \, U_\ell(\hat{\mathfrak g})\to S^\ell(\hat{\mathfrak g})\cong
	U_\ell(\hat{\mathfrak g})/U_{\ell-1}(\hat{\mathfrak g}) 
	$$
	be the canonical projection. For each $\pi\in\mathcal P^\ell $ choose $u(\pi)\in U_\ell(\hat{\mathfrak g})$ such that $p_\ell(u(\pi))=\pi$. Then we have a PBW basis
	$$
	\{u(\pi){\bf 1} \mid \pi\in\mathcal P_{<0}\}
	$$
	of $N(k\Lambda_0)$ and the corresponding spanning set on the quotient $L(k\Lambda_0)$.
	\smallskip
	
	\item  Determine the set of leading terms $ \ell\text{\!\it t\,}(\bar R)\subset\mathcal P$ of relations $r(m)\in\bar R\,\backslash\{0\}$:
	$$
	\rho=\ell\text{\!\it t\,}(r(m)) \quad \text{if} \quad 
	r(m)=c_\rho u(\rho)+\sum_{\rho\prec \kappa}c_\kappa u(\kappa),\quad
	c_{\rho}\neq0.
	$$
	Then we can parameterize a basis of the loop module	 $\bar R$ by its leading terms, i.e. there is a basis
	$\{r(\rho)\mid\rho\in \ell\text{\!\it t\,}(\bar R)\}$ of $\bar R$ such that $\ell\text{\!\it t\,}(r(\rho))=\rho$, i.e.
	\begin{equation}\label{E:r(rho)=}
	r(\rho)= u(\rho)+\sum_{\rho\prec \kappa}c_\kappa u(\kappa).
	\end{equation}
	\item By using relations $r(\rho)|_{L(k\Lambda_0)}=0$ reduce the spanning set $\{u(\pi){\bf 1} \mid \pi\in{\mathcal P}_{<0}\}$ of $L(k\Lambda_0)$ to (hopefully\footnote{
		The obtained spanning set $\mathcal B$ is not necesseraly a basis---see \cite{MP2}.
	}) a basis
	\begin{equation*}
		\mathcal B=\{u(\pi)\cdot {\bf 1}\mid  \pi\in{\mathcal P}_{<0}\,\backslash( \ell\text{\!\it t\,}(\bar R))\}.
	\end{equation*}
	Here $ \pi\in{\mathcal P}_{<0}\,\backslash( \ell\text{\!\it t\,}(\bar R))$ denotes monomials $\pi$ which are not in the ideal $(\bar R)$ generated by relations $\bar R$. 
\end{itemize}

If we think of monomials $\pi$ as colored partitions, then the spanning set of monomial vectors $\mathcal B\subset L(k\Lambda_0)$ is parameterized by partitions which do not contain any subpartition  $\rho\in\ell\text{\!\it t\,}(\bar R)$---this is some sort of combinatorial ``difference conditions'' on parts of the partition $\pi$.

Note that the last two steps in constructing combinatorial basis depend on a choice of order $\prec $ on $\mathcal P$. In some cases the spanning set $\mathcal B$ is a basis and we have the corresponding combinatorial description of the character of $L(k\Lambda_0)$, but to prove linear independence of $\mathcal B$ might be a difficult task. One possible way to prove that $\mathcal B$ is a basis is to have the ``correct'' combinatorial character formula, obtained by some combinatorial method (like in \cite{DK} where the related conjecture in \cite{CMPP} is proved for $k=1$), or any other way. Here we pursue the idea that the correct character formula for $L(k\Lambda_0)$ can be obtained by constructing a combinatorial basis of $\bar RN(k\Lambda_0)$ because the character formula for $N(k\Lambda_0)\cong  S(\hat{\mathfrak g}_{<0})$ is obvious, and
$$
\text{ch\,}L(k\Lambda_0)=\text{ch\,}N(k\Lambda_0) -\text{ch\,}\bar RN(k\Lambda_0).
$$
In order to describe a combinatorial basis of $\bar RN(k\Lambda_0)$ we set
$$
u(\rho\subset\pi)=r(\rho)u(\kappa)
$$
for an embedding $\rho\subset\pi$ such that $\pi=\rho\kappa$,   $\rho\in\ell\text{\!\it t\,}(\bar R)$. 
Note that 
$$
r(\rho)u(\kappa)=u(\pi)+\sum_{\pi\prec \tau}c_\tau u(\tau),
$$
so we have $\ell\text{\!\it t\,}(u(\rho\subset\pi))=\pi$.
With this notation we can write the spanning set of $\bar RN(k\Lambda_0)$ as
\begin{equation}\label{E: Groebner type spanning set}
	u(\rho\subset\pi){\bf 1}, \quad \rho\in \ell\text{\!\it t\,}(\bar R), \pi\in (\ell\text{\!\it t\,}(\bar R))\cap\mathcal P_{<0}.
\end{equation}
If for any two embeddings $\rho_1\subset\pi$ and $\rho_2\subset\pi$ we have a relation
\begin{equation}\label{E: general relations among relations}
	u(\rho_1\subset\pi){\bf 1}-u(\rho_2\subset\pi){\bf 1}
	=\sum_{\pi\prec \pi', \ \rho'\subset\pi'}c_{\rho'\subset\pi'}\,u(\rho'\subset\pi'){\bf 1},
\end{equation}
then we can reduce the spanning set (\ref{E: Groebner type spanning set}) by using (\ref{E: general relations among relations}), and for each $\pi$ we may take just one embedding $\rho(\pi)\subset\pi$, $ \rho(\pi)\in \ell\text{\!\it t\,}(\bar R)$, and the corresponding vector for the reduced spanning set of $\bar RN(k\Lambda_0)$
\begin{equation}\label{E: Groebner type basis}
	u(\rho(\pi)\subset\pi){\bf 1}, \quad  \pi\in (\ell\text{\!\it t\,}(\bar R))\cap\mathcal P_{<0}.
\end{equation}
Such a spanning set is obviously linearly independent since the leading terms are the elements of the PBW basis of $N(k\Lambda_0)$, i.e.
$$
\ell\text{\!\it t\,}(u(\rho(\pi)\subset\pi){\bf 1})=u(\pi){\bf 1}.
$$
The reasoning above applies to all untwisted affine Lie algebras, so {\it relations among relations (\ref{E: general relations among relations})} (would) imply linear independence of the spanning set $\mathcal B$ of $L(k\Lambda_0)$. 
\begin{remark}
If $\pi=\rho_1\rho_2\kappa$, $\rho_1,\rho_2\in\ell\text{\!\it t\,}(\bar R)$, then we have two embeddings $\rho_1 \subset \pi$ and $\rho_2 \subset \pi$ and $\ell(\pi)\geq 2k+2$. From (\ref{E:r(rho)=}) we have
\begin{align*}
	r(\rho_1)r(\rho_2)u(\kappa)&= \big(u(\rho_1)+\sum_{\rho_1\prec \kappa_1}c_{\kappa_1} u(\kappa_1)\big)r(\rho_2)u(\kappa)\\
	&= r(\rho_1)\big(u(\rho_2)+\sum_{\rho_2\prec \kappa_2}c_{\kappa_2} u(\kappa_1)\big)u(\kappa)\\
\end{align*}	
and (\ref{E: general relations among relations}) easily follows. Hence the problem is to check  (\ref{E: general relations among relations}) ``only'' for
$$
k+2\leq\ell(\pi)\leq 2k+1.
$$

For $k=1$ we have $k+2= 2k+1=3$, i.e. we have to check  (\ref{E: general relations among relations}) only for $\ell(\pi)=3$, and this was done in \cite{PS1}.

On the other hand, for $k=2$  we have to check  (\ref{E: general relations among relations}) for $4\leq\ell(\pi)\leq 5$. The main result of this paper---the theorem below---gives (\ref{E: general relations among relations}) only for $\ell(\pi)=4$, and it is hoped that the case $\ell(\pi)=5$ will somehow follow from that\footnote{All necessary relations among relations in \cite{MP1} for all levels follow from Lemma 9.2.}.
\end{remark}

\begin{theorem}\label{T:the main theorem Introduction}
	For any two embeddings
	$\rho_1 \subset \pi$ and $\rho_2 \subset \pi$ in $\pi\in\mathcal P^{4}(m)$,
	where $\rho_1, \rho_2 \in\ell \!\text{{\it t\,}}(\bar{R})$, we
	have a level $2$ relation for $C_n^{(1)}$
	\begin{equation}\label{E:9.2}
		u(\rho_1 \subset \pi) - u(\rho_2 \subset \pi) 
		=\sum_{\pi\prec \pi', \ \rho\subset\pi'}c_{\rho\subset\pi'}\,u(\rho\subset\pi').
	\end{equation}
\end{theorem}

\section{The array of negative root vectors of $C_{n}\sp{(1)}$}

We fix a simple Lie algebra $\mathfrak{g}$ of type $C_n$, $n\geq 2$, i.e.  $\mathfrak{g}=\mathfrak{sp}_{2n}(\mathbb C)$. For a given Cartan subalgebra $\mathfrak h$ and the corresponding
root system $\Delta$ we can write (as in \cite{B})
\begin{equation*}
	\Delta = \{\pm(\varepsilon_i\pm\varepsilon_j) \mid i,j=1,...,n\}\backslash\{0\}
\end{equation*}
with simple roots 
$\alpha_1= \varepsilon_1-\varepsilon_2$,  $\alpha_2=\varepsilon_2-\varepsilon_3$,  \dots  $\alpha_{n-1}=\varepsilon_{n-1}-\varepsilon_{n}$, $\alpha_n = 2\varepsilon_n$.
Then $\theta=2 \varepsilon_1$. For each $\alpha\in\Delta$ we choose a root vector $X_{\alpha}$ such that $[X_{\alpha},X_{-\alpha}]=\alpha^{\vee}$. For root vectors
$X_{\alpha}$ we shall use the following notation:
$$\begin{array}{ccc}
	X_{ij}\quad \text{or just}\quad ij &  \text{if}\   &\alpha =\varepsilon_i + \varepsilon_j\ , \ i\leq j\,,\\
	X_{\underline{i}\underline{j}}\quad \text{or just}\quad \underline{i}\underline{j} & \ \text{if}\  &\alpha =-\varepsilon_i - \varepsilon_j\ , \ i\geq j\,,\\
	X_{i\underline{j}}\quad \text{or just}\quad i \underline{j} & \ \text{if}\  &\alpha =\varepsilon_i - \varepsilon_j\ , \ i\neq j\,.\\
\end{array}
$$
With the previous notation $x_\theta=X_{11}$. We also write for $i=1, \dots, n$
$$
X_{i\underline{i}}=\alpha_i^{\vee}\ \text{or just}\ i\underline{i} \,.
$$
These vectors $X_{ab}$ form a basis $B$ of $\mathfrak g$ which we shall write in a triangular scheme. For example, for $n=3$ the basis $B$ is
\begin{center}
	\begin{tikzpicture} [scale=0.7]
		\node at (0,0) {$11$};\node at (2,0) {$22$};\node at (4,0) {$33$};
		\node at (6,0) {$\underline{3}\underline{3}$};\node at (8,0) {$\underline{2}\underline{2}$};\node at (10,0) {$\underline{1}\underline{1}$};
		\node at (1,1) {$12$};\node at (3,1) {$23$};\node at (5,1) {$3\underline{3}$};
		\node at (7,1) {$\underline{3}\underline{2}$};	\node at (9,1) {$\underline{2}\underline{1}$};
		\node at (2,2) {$13$};\node at (4,2) {$2\underline{3}$};\node at (6,2) {$3\underline{2}$};\node at (8,2) {$\underline{3}\underline{1}$};
		\node at (3,3) {$1\underline{3}$};\node at (5,3) {$2\underline{2}$};\node at (7,3) {$3\underline{1}$};
		\node at (4,4) {$1\underline{2}$};	\node at (6,4) {$2\underline{1}$};
		\node at (5,5) {$1\underline{1}$};
	\end{tikzpicture}
\end{center}	
\begin{center}
	Figure 1
\end{center}
In order to simplify counting of embeddings of leading terms of relations, we will use  the usual matrix notation for the basis $B$, i.e. we will use $i=1,\dots, 2n$ for the first index for rows and $j=1,\dots, 2n$ for the second index for columns/diagonals. 
 For example, for $n=3$ the basis $B$ is
\begin{center}
	\begin{tikzpicture} [scale=0.7]
		\node at (0,0) {$11$};\node at (2,0) {$12$};\node at (4,0) {$13$};
		\node at (6,0) {$14$};\node at (8,0) {$15$};\node at (10,0) {$16$};
		\node at (1,1) {$21$};\node at (3,1) {$22$};\node at (5,1) {$23$};
		\node at (7,1) {$24$};	\node at (9,1) {$25$};
		\node at (2,2) {$31$};\node at (4,2) {$32$};\node at (6,2) {$33$};\node at (8,2) {$34$};
		\node at (3,3) {$41$};\node at (5,3) {$42$};\node at (7,3) {$43$};
		\node at (4,4) {$51$};	\node at (6,4) {$52$};
		\node at (5,5) {$61$};
	\end{tikzpicture}
\end{center}	
\begin{center}
	Figure 2
\end{center}
\bigskip

\noindent
Moreover, we shall write $\bar{B}_{<0}=\coprod_{j>0}{B}\otimes t^{-j}$ in 
the  following scheme
\bigskip

\begin{center}
	\begin{tikzpicture} [scale=0.5]
		\draw (0,0) -- +(10,0) -- +(5,5) -- cycle;
		\node at (5,2) {$B\otimes t^{-1}$};
		\node at (0.8,0.3) {$11$};\node at (8.7,0.3) {$1,2n$};\node at (5,4) {$2n,1$};
		\draw (10.5,0.5) -- +(-5,5) -- +(5,5) -- cycle;
		\node at (11,3.5) {$B\otimes t^{-2}$};
		\node at (10.5,1.3) {$2,2n$};\node at (7,5) {$2n+1,1$};\node at (14,5) {$2n+1,2n$};
		\draw (11,0) -- +(10,0) -- +(5,5) -- cycle;
		\node at (16,2) {$B\otimes t^{-3}$};
		\node at (13,0.3) {$1,2n+1$};\node at (19.5,0.3) {$1,4n$};\node at (16,4) {$2n,2n+1$};
		\node[fill=black, circle, inner sep=1pt] at (20.5,3){};\node[fill=black, circle, inner sep=1pt] at (21,3){};\node[fill=black, circle, inner sep=1pt] at (21.5,3){};
	\end{tikzpicture}
	\end{center}	
\begin{center}
	Figure 3
\end{center}	
which we call {\it the array of negative root vectors of $C_{n}\sp{(1)}$}.
\bigskip

By the main theorem in \cite{PS2}, {\it the monomial
\begin{equation}\label{E:12.12}
	\rho=\prod_{\beta \in \mathcal{B}} X_{\beta}(-j-1)^{m_{\beta,j+1}}\ \prod_{\alpha \in \mathcal{A}} X_{\alpha}(-j)^{m_{\beta,j}}
\end{equation}
is the leading term of the relation $r(\rho)\in\bar R$ for level $k$ standard modules of affine Lie algebra of the type $C_n\sp{(1)}$, where  
\begin{equation}\label{E:12.11}
	\sum_{\beta\in\mathcal{B}}m_{\beta,j+1} =b\quad , \quad \sum_{\alpha\in\mathcal{A}}m_{\alpha,j} =a\quad , \quad	a+b = k+1
\end{equation}
and 
$\mathcal{B}$ and $\mathcal{A}$ are the admissible pair of cascades in degree $-j-1$ and $-j$.}

For a general rank we may visualize the admissible pair of cascades $\mathcal{B}$ and $\mathcal{A}$ as on Figure 4 (see Figure 1 in  \cite{PS2}).
\begin{figure}[h]
	\begin{center}
		\begin{tikzpicture} [scale=0.5]
			\draw (7,0) -- +(0,7) -- +(7,0) -- cycle;
			\draw (7,3) -- (11,3) -- (11,0);
			\draw[red] (9,5) -- +(-0.5,0) -- +(-0.5,-1) -- + (-1.5,-1) -- + (-1.5,-1.5) -- +(-2,-1.5) -- (7,3);
			\node[fill=red, circle, inner sep=2pt](a) at (9,5) {};
			\node[fill=red, circle, inner sep=2pt](a) at (7,3) {};
			\draw[blue] (12,2) -- +(0,-0.5) -- +(-0.5,-0.5) -- + (-0.5,-1.5) -- + (-0.75,-1.5) -- +(-0.75,-2) -- (11,0);
			\node[fill=blue, circle, inner sep=2pt](a) at (12,2) {};
			\node[fill=blue, circle, inner sep=2pt](a) at (11,0) {};
			\node at (11.5,3) {(r,r)};
			\node[red] at (8,4.5) {$\mathcal B$};
			\node[blue] at (12,1) {$\mathcal A$};
		\end{tikzpicture}
	\end{center}
	Figure 4
\end{figure}

With our new way of writing $\bar{B}_{<0}=\coprod_{m<0}{B}\otimes t^m$ on Figure 3, we may reinterpret the pair of admissible cascades $(\mathcal B, \mathcal A)$ on Figure 4 as the
downward zig-zag line on Figure 5.

\begin{figure}[h]
\begin{center}
	\begin{tikzpicture} [scale=0.5, rotate around={45:(7,3.5)}]
		\draw (7,7) -- +(-7,0) -- +(0,-7) -- cycle;
		\draw (4,7) -- (4,3) -- (7,3);
		\draw[blue] (5,2) -- +(0.5,0) -- +(0.5,0.5) -- + (1.5,0.5) -- + (1.5,0.75) -- +(2,0.75) -- (7,3);
		\node[fill=blue, circle, inner sep=2pt] at (5,2) {};
		\draw (7,0) -- +(0,7) -- +(7,0) -- cycle;
		\draw (7,3) -- (11,3) -- (11,0);
		\draw[red] (9,5) -- +(-0.5,0) -- +(-0.5,-1) -- + (-1.5,-1) -- + (-1.5,-1.5) -- +(-2,-1.5) -- (7,3);
		\node[fill=red, circle, inner sep=2pt] at (9,5) {};
		\node[fill=red, circle, inner sep=3pt] at (7,3) {};
		\node[fill=blue, circle, inner sep=2pt] at (7,3) {};
		\node[red] at (8,4.5) {$\mathcal B$};
		\node[blue] at (6,2) {$\mathcal A$};
	\end{tikzpicture}
	
	Figure 5
\end{center}
\end{figure}
Note that Figure 5 consists of the triangle $\Delta$ on Figure 4 together with its mirror image $\Delta'$ with respect to the hypotenuse, but translated along a pair of parallel sides of $\Delta$ and $\Delta'$ so that the other two sides coincide. Of course, the obtained parallelogram is rotated for $\pi/4$.

On $\bar B_{<0}$, written as on Figure 3, we introduce a partial order $\trianglelefteq$ by

\begin{equation}\label{partial order}
	X_{i,j}\trianglelefteq X_{p,r} \quad \text{if}\quad i\in\{1,\dots , p\}\quad \text{and}\quad j\in\{r,p+r-i\}.	
\end{equation}
\bigskip

\noindent
Anotherwords, $b=X_{i,j}$ lies in the cone bellow the vertex $a=X_{p,r}$, as depicted on Figure 6 below:
\begin{center}
	\begin{tikzpicture} [scale=0.5]
		\draw[dashed] (0,0) -- +(10,0) -- +(5,5) -- cycle;
		\draw[dashed] (10.5,0.5) -- +(-5,5) -- +(5,5) -- cycle;
		\draw[dashed] (11,0) -- +(10,0) -- +(5,5) -- cycle;
		\draw (-0.8,-0.3) -- +(22.6,0) --+(16.5,6.2) --+(6.1,6.2) -- cycle;
		
		\node at (9,4) [circle, draw, inner sep=4pt] (A)  {a};
		\node at (10,2) [circle, draw, inner sep=3pt] (B)  {b};
		\draw (5.5,0) -- (A) -- (12.5,0);
		\node at (11.5,3.2) {$b \trianglelefteq a$};
	\end{tikzpicture}\\
	Figure 6
\end{center}
\bigskip

\noindent
With this conventions we can restate the description of leading terms $\rho$ of relations $r(\rho)\in\bar R$  in \cite{PS2} (cf. \cite{P}) as: {\it for any zig-zag downward line
$$
a_1\vartriangleright a_2\vartriangleright\dots \vartriangleright a_r
$$
of $r$ points in $\bar B_{<0}$, $1\leq r\leq k+1$, the monomial
\begin{equation}\label{E:the leading terms of relations}
\rho=a_1^{m_1}a_2^{m_2}\dots a_r^{m_r}, \quad m_1+m_2+\dots+m_r=k+1,
\end{equation}
is the leading term of the relation $r(\rho)\in\bar R$ for level $k$ standard modules of affine Lie algebra of the type $C_n\sp{(1)}$. Moreover, these are all leading terms of $\bar R$ in $\mathcal P_{<0}$.}
\bigskip

In the next section we shall count the number of embeddings $\rho\subset\pi$ for $\text{supp\,}\pi $ in three successive triangles, i.e. in the trapezoid\footnote{Note that the position of the trapezoid in Figure 6 is in accord with Figure 3 only when the middle triangle is $B\otimes t^j$ for $j$ even, and for $j$ odd the figure should be flipped. However, in our arguments this will make no difference because the flipped zig-zag downward line is again a zig-zag downward line.} $T$ on Figure 6.
\bigskip

\section{Colored partitions allowing at least two embeddings }

For a colored partition
\begin{equation}\label{E:colored partition}
\pi=\prod\sb{a\in\bar B_{<0}} a^{\pi(a)}
\end{equation}
we have  $|\text{supp\,}\pi |\leq \ell(\pi)$. The basis $\bar B_{<0}$ is writen as the array of $2n+1$ rows and $a_1\vartriangleright\dots\vartriangleright a_r$ implies $r\leq 2n+1$.
\begin{lemma}\label{L: klasifikacija dva ulaganja}
Let $\ell(\pi)=k+2$ and assume that $\pi$ allows two embeddings of leading terms of relations for level $k$ standard modules. Then $\text{supp\,}\pi $ is one of the following types (for $r, s\in\mathbb N$ and $\delta$ in the set of two symbols $\vert$ and $\vert\vert$):
\begin{itemize}
\item[$(A_{r})$] $\text{supp\,}\pi =\{a_1, \dots, a_r\}$, $r\geq 2$, $a_1\vartriangleright\dots\vartriangleright a_r$.
\item[$(B_{r\,\delta})$] $\text{supp\,}\pi =\{a_1, \dots, a_r,b,c\}$, $r\geq 1$, $a_1\vartriangleright\dots\vartriangleright a_r$, $a_r\vartriangleright b$, $a_r\vartriangleright c$ and $b$ and $c$ are not comparable. We set $\delta$ to be $\vert$  if $b$ and $c$ are in the same row, and $\vert\vert $ otherwise.
\item[$(C_{\delta\, r})$] $\text{supp\,}\pi =\{b,c,a_1, \dots, a_r\}$, $r\geq 1$, $a_1\vartriangleright\dots\vartriangleright a_r$, $b\vartriangleright a_1$, $c\vartriangleright a_1$ and $b$ and $c$ are not comparable. We set $\delta$ to be $\vert$  if $b$ and $c$ are in the same row, and $\vert\vert $ otherwise.
\item[$(D_{r \delta\, s})$] $\text{supp\,}\pi =\{a_1, \dots, a_r,b,c,d_1,\dots,d_s\}$, $r,s\geq 1$, $a_1\vartriangleright\dots\vartriangleright a_r$, $a_r\vartriangleright b\vartriangleright d_1$, $a_r\vartriangleright c\vartriangleright d_1$, $d_1\vartriangleright\dots\vartriangleright d_s$, and $b$ and $c$ are not comparable. We set $\delta$ to be $\vert$  if $b$ and $c$ are in the same row, and $\vert\vert $ otherwise.
\end{itemize}
\end{lemma}
\begin{proof}
Let $\pi =a_1^{f_1}\dots a_r^{f_r}$, $f_1\cdots f_r>0$, $f_1+\dots +f_r=k+2$, $a_1\vartriangleright\dots\vartriangleright a_r$. If in $\pi$ we erase one factor $a_j$, $j\in\{1,\dots,r\}$, we get $r$ leading terms of relations
$$
\rho_1=a_1^{f_1-1}\dots a_r^{f_r}, \quad\dots\quad \rho_r=a_1^{f_1}\dots a_r^{f_r-1}
$$
and we have $r$ embeddings
$\rho_1\subset \pi$, \dots, $\rho_r\subset \pi$. In particular, for $r\geq 2$ we have at least two embeddings in $\pi$.

Now let  $\ell(\pi)=k+2$ and let $\rho_1\subset \pi$ and $\rho_2\subset \pi$ be two different embeddings of leading terms of relations,   $\ell(\rho_1)=\ell(\rho_2)=k+1$. Then there are parts $c_1$ and  $c_2$ of $\pi$ such that
$$
\pi=c_1\rho_1=c_2\rho_2=c_1c_2(\rho_1\cap\rho_2),
$$
and $\rho_1\neq\rho_2$ implies  $c_1\neq c_2$. Hence 
$$
\rho_1=c_2(\rho_1\cap\rho_2), \quad \rho_2=c_1(\rho_1\cap\rho_2).
$$
Since $\rho_1$ and $\rho_2$ are the leading terms of relations, (\ref{E:the leading terms of relations}) implies that the partition $\rho_1\cup\rho_2$ can be written as 
$$
  a_1\trianglerighteq\dots\trianglerighteq a_t
$$
and  $\rho_1$ and $\rho_2$ as
\begin{equation}\label{E: ro1 i ro2 u klasifikaciji}
\begin{aligned}
\rho_1&\colon \quad  a_1\trianglerighteq\dots\trianglerighteq a_p \trianglerighteq c_2\trianglerighteq a_{p+1} \trianglerighteq\dots\trianglerighteq a_t,\\
\rho_2&\colon \quad  a_1\trianglerighteq\dots\trianglerighteq a_q \trianglerighteq c_1\trianglerighteq a_{q+1} \trianglerighteq\dots\trianglerighteq a_t.
\end{aligned}
\end{equation}
If $q<p$, then 
$$
 a_q \trianglerighteq c_1\trianglerighteq a_{q+1} \trianglerighteq\dots\trianglerighteq c_2
$$
and $\text{supp\,}\pi $ is of the type $A_{r}$ for some $r\geq2$.

If $q=p$ and $c_1$ and $c_2$ are comparable, say $c_2\vartriangleright c_1$, then $ a_p \trianglerighteq c_2\trianglerighteq a_{p+1}$ and $ a_p \trianglerighteq c_1\trianglerighteq a_{p+1}$ implies
$$
 a_p \trianglerighteq c_2\vartriangleright c_1\trianglerighteq a_{p+1},
$$
and again  $\text{supp\,}\pi $ is of the type $A_{r}$ for some $r\geq2$. Hence the lemma follows and  $\text{supp\,}\pi $ is of the type $B_{r\,\delta}$, $C_{\delta\, r}$ or $D_{r \delta\, s}$, depending on the position of incomparable parts $c_1$ and $c_2$ in $\rho_1$ and $\rho_2$ written as (\ref{E: ro1 i ro2 u klasifikaciji}) for $p=q$.

The distinction, and the notation of relative position of parts $c_1$ and $c_2$, i.e. $\delta = ||$ if  $c_1$ and $c_2$ are in different rows, and $\delta = |$ otherwise, will be convenient later on.
\end{proof}

\begin{lemma}\label{L: broj particija za razna ulaganja}
(i) Let $2\leq r\leq \min\{2n+1, k+2\}$. Then for any choice of $a_1\vartriangleright\dots\vartriangleright a_r$ there is 
$$
{k+1\choose r-1}
$$ 
partitions $\pi$ of the form
\begin{equation}\label{E: broj particija za isti nosac}
\pi =a_1^{f_1}\dots a_r^{f_r},\qquad f_1\cdots f_r>0, \quad f_1+\dots +f_r=k+2
\end{equation}
and for each such $\pi$ there are $r$ embeddings $\rho_1\subset \pi$, \dots, $\rho_r\subset \pi$ for
$$
\rho_1=a_1^{f_1-1}\dots a_r^{f_r}, \quad\dots\quad \rho_r=a_1^{f_1}\dots a_r^{f_r-1}.
$$

(ii) Let $r,s\geq0$, $3\leq r+2+s \leq  k+2$, $3\leq r+a+s \leq  2n+1$ with $a=1$ if $\delta$ is $|$ and $a=2$ if $\delta$ is $||$. For any choice of $\{a_1,\dots,a_r, c,d,b_1,\dots,b_s\}$ of the type $B_{r\,\delta}$ (for $s=0$), $C_{\delta\, s}$ (for $r=0$) or $(D_{r \delta\, s})$ there is 
$$
{k-1\choose s+r-1}
$$ 
partitions $\pi$ of the form
\begin{equation}\label{E: broj particija za isti nosac}
\pi =a_1^{f_1}\dots a_r^{f_r}cd\, b_1^{g_1}\dots b_s^{g_s},\quad 
f_1\cdots f_r g_1\cdots g_s>0, \  f_1+\dots +f_r+ g_1+\dots +g_s=k,
\end{equation}
and for each such $\pi$ there are $2$ embeddings $\rho_1\subset \pi$ and $\rho_2\subset \pi$ for
$$
\rho_1=a_1^{f_1}\dots a_r^{f_r}c\, b_1^{g_1}\dots b_s^{g_s}, \quad\quad
\rho_2=a_1^{f_1}\dots a_r^{f_r}d\, b_1^{g_1}\dots b_s^{g_s}.
$$
\end{lemma}
\begin{proof} Since there is ${k+1\choose r-1}$ ways of writing $f_1+\dots +f_r=k+2$,  $f_1\cdots f_r>0$, the first statement is clear, and the second statement is obvious since $\ell(\rho)=k+1$. This proves (i), and the proof of (ii) is similar.
\end{proof}

As in \cite{PS1} for a colored partition $\pi$ set
$$
N(\pi)=\max\{\#\mathcal E(\pi)-1,0\}\quad \mathcal E(\pi)=\{\rho\in \ell\text{\!\it t\,}(\bar R)\mid \rho\subset\pi\}.
$$
Let $T$ be a trapezoid consisting of three consecutive triangles in the array ${\bar B}_{<0}$.
Set
$$
\Sigma_T(A_r)=\sum_{S\subset T,\ S\text{ is of the type }A_r}1,\qquad
N_T(A_r)=\sum_{\pi, \, \text{supp\,}\pi\subset T, \,\text{supp\,}\pi\text{ is of the type }A_r}N(\pi),
$$
and likewise for types  $B_{r\,\delta}$, $C_{\delta\, r}$ and $D_{r \delta\, s}$. 

Note that two diferent orientations for trapezoids are possible: with a long base down, say $T$, and with the long base up, say $T'$. however we have an ``up-down symmetry'' so that $\Sigma_T(A_r)=\Sigma_{T'}(A_r)$ and
$$
\Sigma_{T}(B_{r\,\delta})=\Sigma_{T'}(C_{\delta\, r}), \quad 
\Sigma_{T}(C_{\delta\, r})=\Sigma_{T'}(B_{r\,\delta}), \quad 
\Sigma_{T}(D_{r \,\delta\, s})=\Sigma_{T'}(D_{s\,\delta\, r}),
$$
so we shall consider only a trapezoid oriented as $T$. clearly Lemma \ref{L: broj particija za razna ulaganja} implies:
\begin{lemma}\label{L: broj potrebnih relacija} For $r$ and $s$ as in Lemma \ref{L: broj particija za razna ulaganja} we have:
\begin{enumerate}
\item $N_T(A_r)=(r-1){k+1\choose r-1}\,\Sigma_T(A_r)$,
\item $N_T(B_{r\,\delta})={k-1\choose r-1}\,\Sigma_{T}(B_{r\,\delta})$,
\item $N_T(C_{\delta\, r})={k-1\choose r-1}\,\Sigma_{T}(C_{\delta\, r})$,
\item $N_T(D_{r \,\delta\, s})={k-1\choose s+r-1}\,\Sigma_{T}(D_{r \,\delta\, s})$.
\end{enumerate}
\end{lemma}

\begin{lemma}\label{L: suma za tip A} For $2\leq r\leq 2n+1$ we have:
$$
\begin{aligned}
\Sigma_T(A_r)&=\sum_{i_1=r}^{2n+1}\sum_{j_1=1}^{4n+1-i_1}\sum_{i_2=r-1}^{i_1-1}\sum_{j_2=j_1}^{j_1+i_1-i_2}
\dots\sum_{i_{p}=r-p+1}^{i_{p-1}-1}\sum_{j_{p}=j_{p-1}}^{j_{p-1}+i_{p-1}-i_{p}}\dots
\sum_{i_{r}=1}^{i_{r-1}-1}\sum_{j_{r}=j_{r-1}}^{j_{r-1}+i_{r-1}-i_{r}}1\\
&=\sum_{i_1=r}^{2n+1}\sum_{i_2=r-1}^{i_1-1}\cdots\sum_{i_{p}=r-p+1}^{i_{p-1}-1}
\dots\sum_{i_{r}=1}^{i_{r-1}-1}\pi(i_1,\dots,i_r),\\
&\pi(i_1,\dots,i_r)=(4n+1-i_1)(i_1-i_2+1)\cdots(i_{p-1}-i_{p}+1)\cdots(i_{r-1}-i_{r}+1).
\end{aligned}
$$
\end{lemma}

\begin{proof}
We want to count all
$$
S=\{a_1,\dots,a_r\}\subset T, \quad a_1\vartriangleright \dots \vartriangleright a_r.
$$
The points $a_p=(i_p,j_p)\in T$ are in distinct rows $i_1>\dots>i_r$. For the first point $a_1$ in the row $i_1$ we may choose $j_1\in\{1,\dots,4n+1-i_1\}$, altogether $4n+1-i_1$ points. For the second point $a_2\vartriangleright a_1$ in the row $i_2$ we may choose $j_2\in\{j_1,\dots,j_1+i_1-i_2\}$, altogether $i_1-i_2+1$ points. In this way we continue till the point $a_r\in T$.
\end{proof}

\begin{remark}\label{R: ostali slucajevi} The types  $B_{r\,\delta}$, $C_{\delta\, s}$ and $D_{r \delta\, s}$ differ from the type $A_r$ by having two incomparable points $b$ and $c$. In the case when $b$ and $c$ are on the same line $\ell$ (i.e. $\delta$ is $|$) in our counting procedure   for $\Sigma_T$ we first determine all possible distances $d$ between $b$ and $c$, and then all possible lines $\ell$ through $b$ and $c$---we sketch the three corresponding configurations as: 
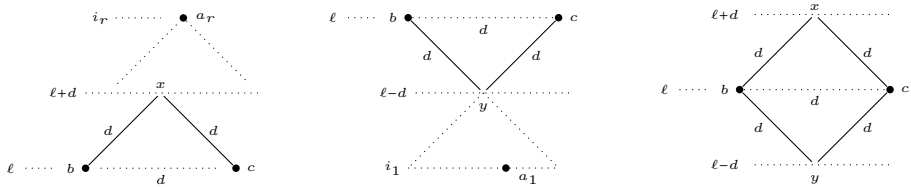
\begin{center}
\begin{tikzpicture} [scale=1]
              \node (d) at (2.3,1) {};
		\node[fill=black, circle, inner sep=1pt](e) at (3.3,2) {};
		\node (f) at (4.3,1) {};
		\draw[dotted] (d) -- (e);
		\draw[dotted] (e) -- (f);
		
	       \node at (2.2,2) {$\scriptscriptstyle i_r$};
		\draw[dotted] (2.4,2) -- (3,2);
		\node at (3.6,2) {$\scriptscriptstyle a_r$};
		\node at (3,1.1) {$\scriptscriptstyle x$};
		\draw[dotted] (2,1) -- (4.3,1);
		\node at (1.7,1) {$\scriptscriptstyle \ell+d$};
		
		\node[fill=black, circle, inner sep=1pt](a) at (2,0) {};
		\node at (1.8,0) {$\scriptscriptstyle b$};
		\node [circle, inner sep=1pt](b) at (3,1) {};
		\node at (4.2,0) {$\scriptscriptstyle c$};
		\node at (1,0) {$\scriptscriptstyle \ell$};
		\draw[dotted] (1.2,0) -- (1.6,0);

		\node [fill=black, circle, inner sep=1pt](c) at (4,0) {};
		\draw (a) -- (b) -- (c);
		\node at (3,-0.15) {$\scriptscriptstyle d$};
		\node at (2.3,0.5) {$\scriptscriptstyle d$};
		\node at (3.7,0.5) {$\scriptscriptstyle d$};
		\draw[dotted] (2.2,0) -- (3.8,0);
	
\end{tikzpicture}\qquad
\begin{tikzpicture} [scale=1]
              \node (d) at (2.3,-1) {};
		\node[fill=black, circle, inner sep=1pt](e) at (3.3,-2) {};
		\node at (3.6,-2.15) {$\scriptscriptstyle a_1$};
		
		\node (f) at (4.3,-1) {};
		\draw[dotted] (2.1,-1) -- (f);
		\node at (1.8,-1) {$\scriptscriptstyle \ell-d$};
		
		\draw[dotted] (2,-2) -- (4,-2);
		\node at (1.8,-2) {$\scriptscriptstyle i_1$};
		
		\node [fill=black, circle, inner sep=1pt](b) at (2,0) {};
		\node at (1.8,0) {$\scriptscriptstyle b$};
		\node[circle, inner sep=1pt](y) at (3,-1) {};
		\node at (3,-1.2) {$\scriptscriptstyle y$};
		
		\draw[dotted] (3,-1) -- (4,-2);
		\draw[dotted] (3,-1) -- (2,-2);
		
		\node at (2.25,-0.5) {$\scriptscriptstyle d$};
		\node at (3.7,-0.5) {$\scriptscriptstyle d$};

		\node [fill=black, circle, inner sep=1pt](c) at (4,0) {};
		\node at (4.2,0) {$\scriptscriptstyle c$};
		\draw (b) -- (y) -- (c);
		\node at (3,-0.15) {$\scriptscriptstyle d$};
		\draw[dotted] (b) -- (c);
		 \node at (1,0) {$\scriptscriptstyle \ell$};
		\draw[dotted] (1.2,0) -- (1.6,0);
\end{tikzpicture}
\qquad
\begin{tikzpicture} [scale=1]
              \node [circle, inner sep=1pt](x) at (3,1) {};
		\node at (1.8,0) {$\scriptscriptstyle b$};
		\node [fill=black, circle, inner sep=1pt](b) at (2,0) {};

		\node[circle, inner sep=1pt](y) at (3,-1) {};
		\node at (4.2,0) {$\scriptscriptstyle c$};
		\node [fill=black, circle, inner sep=1pt](c) at (4,0) {};

		\draw (b) -- (y) -- (c);
		\draw (b) -- (x) -- (c);
		\draw[dotted] (2.2,1) -- (4,1);
		\draw[dotted] (2.2,-1) -- (4,-1);
	       \node at (1,0) {$\scriptscriptstyle \ell$};
		\draw[dotted] (1.2,0) -- (1.6,0);
		\draw[dotted] (b) -- (c);
		\node at (3,-0.15) {$\scriptscriptstyle d$};
		
		\node at (2.25,-0.5) {$\scriptscriptstyle d$};
		\node at (3.7,-0.5) {$\scriptscriptstyle d$};
		
		\node at (2.25,0.5) {$\scriptscriptstyle d$};
		\node at (3.7,0.5) {$\scriptscriptstyle d$};
		
		\node at (1.8,-1) {$\scriptscriptstyle \ell-d$};
		\node at (1.8,1) {$\scriptscriptstyle \ell+d$};
		
		\node at (3,1.1) {$\scriptscriptstyle x$};
		\node at (3,-1.2) {$\scriptscriptstyle y$};
\end{tikzpicture}
\end{center}
After that we argue in a manner very similar to the proof of Lemma \ref{L: suma za tip A}. 
\end{remark}

\begin{lemma}\label{L: suma za tip B$|$} For $1\leq r\leq 2n$ we have:
$$
\begin{gathered}
\Sigma_T(B_{r\, |})=\sum_{d=1}^{2n+1-r}\ \sum_{\ell=1}^{2n+2-d-r}
\ \sum_{i_1=\ell+d+r-1}^{2n+1}\ \sum_{i_2=\ell+d+r-2}^{i_1-1}\ \dots
\ \sum_{i_{r}=\ell+d}^{i_{r-1}-1}\ \pi(d,\ell; i_1,\dots,i_r),\\
\pi(d,\ell; i_1,\dots,i_r)=(4n+1-i_1)(i_1-i_2+1)\cdots(i_{r-1}-i_{r}+1)(i_{r}-\ell-d+1).
\end{gathered}
$$
\end{lemma}

\begin{proof} We should determine the number of all possible  $\{a_1, \dots, a_r,b,c\}\subset T$, $r\geq 1$, such that $a_1\vartriangleright\dots\vartriangleright a_r$, $a_r\vartriangleright b$, $a_r\vartriangleright c$ and $b\neq c$ are on the same line $\ell$. The distance $d=j_2-j_1$ between the points $b=(\ell,j_1)$ and $c=(\ell,j_2)$, $j_2>j_1$, may be $1$ because for $\ell=1$ we still can place $r\leq 2n$ points $a_1\vartriangleright\dots\vartriangleright a_r$ above both $b$ and $c$. In general for $d\geq 1$ and $b$ and $c$ on the line $1\leq\ell\leq 2n+1$ the point $a_r$ is on the line $i_r\geq \ell+d$  and we need at least $r-1$ lines above to place the remaining $r-1$ points $a_1\vartriangleright\dots\vartriangleright a_{r-1}$. Hence
\begin{equation}\label{E:interval za d slucaj Bh=0}
\ell+d+r-1\leq 2n+1,
\end{equation}
and for $\ell=1$ this implies 
$$
1\leq d\leq 2n+1-r.
$$
Once we have fixed $d$, by (\ref{E:interval za d slucaj Bh=0}) we should take
$$
1\leq \ell\leq 2n+2-d-r.
$$
Now we proceed as in the proof of Lemma \ref{L: suma za tip A}, the only difference being that for the vertex $x$ in the equilateral triangle $\Delta(b,c,x)$ there is
$$
i_{r}-\ell-d+1
$$
places on the line $\ell+d$ which are in the cone below $a_r$.
\end{proof}

\begin{lemma}\label{L: suma za tip C$|$} For $1\leq r\leq 2n$ we have:
$$
\begin{gathered}
\Sigma_T(C_{|\, r})= 1\sum_{d=1}^{2n+1-r}\ \sum_{\ell=d+r}^{2n+1}
\ \sum_{i_1=r}^{\ell-d}\ \sum_{i_2=r-1}^{i_1-1}\ \dots
\ \sum_{i_{r}=1}^{i_{r-1}-1}\ \pi(d,\ell; i_1,\dots,i_r),\\
\pi(d,\ell; i_1,\dots,i_r)=(4n+1-\ell-d)(\ell-d-i_1+1)(i_{1}-i_{2}+1)\dots(i_{r-1}-i_{r}+1).
\end{gathered}
$$
\end{lemma}

\begin{proof} 
We should determine the number of all possible  $\{b,c,a_1, \dots, a_r\}\subset T$, $r\geq 1$, such that $a_1\vartriangleright\dots\vartriangleright a_r$, $b\vartriangleright a_1$, $c\vartriangleright a_1$ and $b\neq c$ are on the same line $\ell$. The distance $d=j_2-j_1$ between the points $b=(\ell,j_1)$ and $c=(\ell,j_2)$, $j_2>j_1$, may be $1$ because on the ``smallest'' top row of trapezoid $T$ we have at least $2$ points. For $r$ points $a_1, \dots, a_r$ we need $r$ distinct rows, and $\ell$-th row is above them at the distance $d$. Hence
$$
d+r\leq 2n+1.
$$
For fixed $d$ we clearly have
$$
d+r\leq\ell\leq 2n+1.
$$
There are $4n+1-\ell-d$ ways to put two points $b$ and $c$ on the line $\ell$ at distance $d$.
For fixed $d$, and $b$ and $c$ on the line $1\leq\ell\leq 2n+1$, the point $a_1$ is on the line $i_1\leq \ell-d$. Now we proceed as in the proof of Lemma \ref{L: suma za tip A}, the only difference being that for the vertex $y$ in the equilateral triangle $\Delta(b,y,c)$ there is
$$
\ell-d-i_1+1
$$
places on the line $i_1$ which are in the cone below $y$.
\end{proof}

\begin{lemma}\label{L: suma za tip D$|$} For $2\leq r+s\leq 2n$ we have:
$$
\begin{gathered}
\Sigma_T(D_{r\,|\, s})=\sum_{d=1}^{ \lfloor \tfrac{2n+2-r-s}{2}\rfloor}\ 
 \sum_{\ell=s+d}^{2n+2-r-d}
\ \sum_{i_1=\ell+d-r+1}^{2n+1}\ \sum_{i_2=\ell+d-r+2}^{i_1-1}\ \dots
\ \sum_{i_{r}=\ell+d}^{i_{r-1}-1}\\
\ \sum_{l_1=s}^{\ell-d}\ \sum_{l_2=s-1}^{l_1-1}\ \dots
\ \sum_{l_{s}=1}^{l_{s-1}-1}\ \pi(d,\ell; i_1,\dots,i_r;l_1,\dots,l_s),\\
\pi(d,\ell; i_1,\dots,i_r;l_1,\dots,l_s)=(4n+1-i_1)(i_{1}-_{2}+1)\dots(i_{r-1}-i_{r}+1)
(i_r-\ell-d+1)\\
(\ell-d-l_1+1)(l_{1}-l_{2}+1)\dots(l_{s-1}-l_{s}+1).
\end{gathered}
$$
\end{lemma}

\begin{proof} 
We should determine the number of all possible  $\{a_1, \dots, a_r,b,c,d_1,\dots,d_s\}$, $r,s\geq 1$, $a_1\vartriangleright\dots\vartriangleright a_r$, $a_r\vartriangleright b\vartriangleright d_1$, $a_r\vartriangleright c\vartriangleright d_1$, $d_1\vartriangleright\dots\vartriangleright d_s$,  and $b\neq c$ are on the same line $\ell$. We need to place $r$ distinct rows above $\ell$ at the distance $d$, and $s$ distinct rows below $\ell$ at the distance $d$. Hence $r-1+s-1+2d\leq 2n+1$, and for positive integer $d$ we have
$$
1\leq d\leq \lfloor \tfrac{2n+3-r-s}{2}\rfloor.
$$
For fixed $d$ we have $s+d\leq \ell\leq 2n+1-r-d$, and we proceed as in the proof of lemmas above: 
the factors in $\pi(d,\ell; i_1,\dots,i_r;l_1,\dots,l_s)$ count the number of points we can place in one row when a point in the row above is chosen. In particular, the factor $(i_r-\ell-d+1)$ counts the number of places the vertex $x$ of the rhombus $R(b,y,c,x)$ can take on the line $\ell+d$ in the cone below $a_r$, and 
$(\ell-d-l_1+1)$ is the number of places on the line $l_1$ which are in the cone below the vertex $y$.
\end{proof}

\begin{remark}\label{R: ostali slucajevi za drugi delta} The types  $B_{r\,||}$, $C_{||\, s}$ and $D_{r \,||\, s}$ have two incomparable points $b$ and $c$ in different lines in two possible configurations
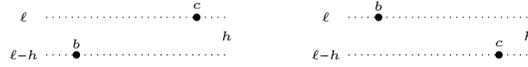
\begin{center}
\begin{tikzpicture} [scale=1]
             \draw[dotted] (1.8,0) -- (4.2,0);
              \node at (2.2,0.15) {$\scriptscriptstyle b$};
             \node[fill=black, circle, inner sep=1pt](e) at (2.2,0) {};
             \node at (1.5,0) {$\scriptscriptstyle \ell-h$};
             
                \node at (4.2,0.25) {$\scriptscriptstyle h$};
             
               \draw[dotted] (1.8,0.5) -- (4.2,0.5);
              \node at (3.8,0.65) {$\scriptscriptstyle c$};
             \node[fill=black, circle, inner sep=1pt](e) at (3.8,0.5) {};
             \node at (1.5,0.5) {$\scriptscriptstyle \ell$};
\end{tikzpicture}
\qquad
\begin{tikzpicture} [scale=1]
             \draw[dotted] (1.8,0.5) -- (4.2,0.5);
              \node at (2.2,0.65) {$\scriptscriptstyle b$};
             \node[fill=black, circle, inner sep=1pt](e) at (2.2,0.5) {};
             \node at (1.5,0.5) {$\scriptscriptstyle \ell$};
             
                \node at (4.2,0.25) {$\scriptscriptstyle h$};
             
               \draw[dotted] (1.8,0) -- (4.2,0);
              \node at (3.8,0.15) {$\scriptscriptstyle c$};
             \node[fill=black, circle, inner sep=1pt](e) at (3.8,0) {};
             \node at (1.5,0) {$\scriptscriptstyle \ell-h$};
\end{tikzpicture}
\end{center}
Due to the left-right symmetry it is enough to consider $\Sigma_T$ for only the first one counted twice. As before, we first determine all possible distances $d=j_2-j_1\geq1$ between $b=(\ell-h,j_1)$ and $c=(j_2,\ell)$, and then all possible $h$ and lines $\ell$ through $c$---we sketch the three corresponding configurations as: 
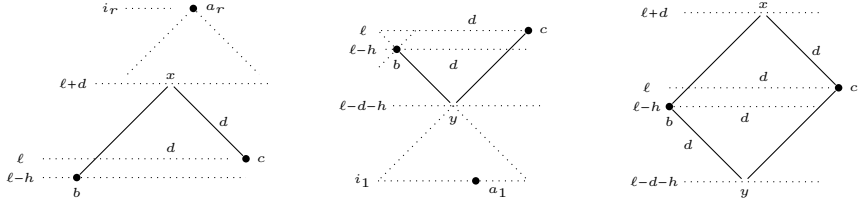
\begin{center}
\begin{tikzpicture} [scale=1]
              \node (d) at (2.3,1) {};
		\node[fill=black, circle, inner sep=1pt](e) at (3.3,2) {};
		\node (f) at (4.3,1) {};
		\draw[dotted] (d) -- (e);
		\draw[dotted] (e) -- (f);
		
	       \node at (2.2,2) {$\scriptscriptstyle i_r$};
		\draw[dotted] (2.4,2) -- (3,2);
		\node at (3.6,2) {$\scriptscriptstyle a_r$};
		\node at (3,1.1) {$\scriptscriptstyle x$};
		\draw[dotted] (2,1) -- (4.3,1);
		\node at (1.7,1) {$\scriptscriptstyle \ell+d$};

		\node[fill=black, circle, inner sep=1pt](a) at (1.75,-0.25) {};
		\node at (1.75,-0.45) {$\scriptscriptstyle b$};
		\node at (1,-0.25) {$\scriptscriptstyle \ell-h$};
		
		\node [circle, inner sep=1pt](b) at (3,1) {};
		\node at (4.2,0) {$\scriptscriptstyle c$};
		\node at (1,0) {$\scriptscriptstyle \ell$};
		\draw[dotted] (1.3,0) -- (3.8,0);
		\draw[dotted] (1.3,-0.25) -- (4,-0.25);

		\node [fill=black, circle, inner sep=1pt](c) at (4,0) {};
		\draw (a) -- (b) -- (c);
		\node at (3,0.15) {$\scriptscriptstyle d$};
		\node at (3.7,0.5) {$\scriptscriptstyle d$};
	
\end{tikzpicture}\qquad
\begin{tikzpicture} [scale=1]
              \node (d) at (2.3,-1) {};
		\node[fill=black, circle, inner sep=1pt](e) at (3.3,-2) {};
		\node at (3.6,-2.15) {$\scriptscriptstyle a_1$};
		
		\node (f) at (4.3,-1) {};
		\draw[dotted] (2.2,-1) -- (f);
		\draw[dotted] (2.1,-0.25) -- (4,-0.25);
		\node at (1.8,-1) {$\scriptscriptstyle \ell-d-h$};
		\node at (1.8,-0.25) {$\scriptscriptstyle \ell-h$};

		\draw[dotted] (2,-2) -- (4,-2);
		\node at (1.8,-2) {$\scriptscriptstyle i_1$};
		
		\node [fill=black, circle, inner sep=1pt](b) at (2.25,-0.25) {};
		\node at (2.25,-0.45) {$\scriptscriptstyle b$};
		\node at (1.8,0) {$\scriptscriptstyle \ell$};
		\node[circle, inner sep=1pt](y) at (3,-1) {};
		\node at (3,-1.2) {$\scriptscriptstyle y$};
		
		\draw[dotted] (3,-1) -- (4,-2);
		\draw[dotted] (3,-1) -- (2,-2);

		\node [fill=black, circle, inner sep=1pt](c) at (4,0) {};
		\node at (4.2,0) {$\scriptscriptstyle c$};
		\draw (b) -- (y) -- (c);
		\node at (3,-0.45) {$\scriptscriptstyle d$};
		\draw[dotted] (b) -- (2,0) -- (c);
		\draw[dotted] (2.5,0) -- (2,-0.5);
		\node at (3.25,0.15) {$\scriptscriptstyle d$};
\end{tikzpicture}
\qquad
\begin{tikzpicture} [scale=1]
              \node [circle, inner sep=1pt](x) at (3.25,1.25) {};
	
	       \node at (2,-0.2) {$\scriptscriptstyle b$};
		\node at (1.7,0) {$\scriptscriptstyle \ell-h$};
		\node [fill=black, circle, inner sep=1pt](b) at (2,0) {};

		\node[circle, inner sep=1pt](y) at (3,-1) {};
		\node at (4.45,0.25) {$\scriptscriptstyle c$};
		\node [fill=black, circle, inner sep=1pt](c) at (4.25,0.25) {};
		\draw[dotted] (2,0.25) -- (c);
		\node at (1.7,0.25) {$\scriptscriptstyle \ell$};
		
		\draw (b) -- (y) -- (c);
		\draw (b) -- (x) -- (c);
		\draw[dotted] (2.2,1.25) -- (4,1.25);
		\draw[dotted] (2.2,-1) -- (4,-1);
		
		\draw[dotted] (b) -- (4,0);
		\node at (3,-0.15) {$\scriptscriptstyle d$};
		
		\node at (2.25,-0.5) {$\scriptscriptstyle d$};

		\node at (3.25,0.4) {$\scriptscriptstyle d$};
		\node at (3.95,0.75) {$\scriptscriptstyle d$};
		
		\node at (1.8,-1) {$\scriptscriptstyle \ell-d-h$};
		\node at (1.8,1.25) {$\scriptscriptstyle \ell+d$};
		
		\node at (3.25,1.35) {$\scriptscriptstyle x$};
		\node at (3,-1.15) {$\scriptscriptstyle y$};
\end{tikzpicture}
\end{center}
\end{remark}

\begin{lemma}\label{L: suma za tip B$||$}For $1\leq r\leq 2n-1$ we have:
$$
\begin{gathered}
\Sigma_T(B_{r\, ||})=2\sum_{d=1}^{2n-r}\ \sum_{h=1}^{2n+1-d-r}\ \sum_{\ell=h+1}^{2n+2-d-r}
\sum_{i_1=\ell+d+r-1}^{2n+1}\ \sum_{i_2=\ell+d+r-2}^{i_1-1}\ \dots
\ \sum_{i_{r}=\ell+d}^{i_{r-1}-1}\ \pi(d,\ell; i_1,\dots,i_r),\\
\pi(d,\ell; i_1,\dots,i_r)=(4n+1-i_1)(i_1-i_2+1)\cdots(i_{r-1}-i_{r}+1)(i_{r}-\ell-d+1).
\end{gathered}
$$
\end{lemma}

\begin{proof} We argue as in the proof of Lemma \ref{L: suma za tip B$|$}: 
for $d\geq 1$, $h\geq1$ and $c$ on the line $1\leq\ell\leq 2n+1$ the point $a_r$ is on the line $i_r\geq \ell+d$  and we need at least $r-1$ lines above to place the remaining $r-1$ points $a_1\vartriangleright\dots\vartriangleright a_{r-1}$. Hence
\begin{equation}\label{E:interval za d slucaj Bh>0}
\ell+d+r-1\leq 2n+1,
\end{equation}
and for $\ell=h+1$ this implies $d\leq 2n+1-r-h$. For $h=1$ we determine the bounds for $d$, for chosen $d$ we determine the bounds for $h$, and by using (\ref{E:interval za d slucaj Bh>0}) we determine the bounds for $\ell$:
$$
1\leq d\leq 2n-r,\quad 1\leq h\leq 2n+1-d-r,\quad h+1\leq \ell\leq 2n+2-d-r.
$$
Now we proceed as in the proof of Lemma \ref{L: suma za tip B$|$}.
\end{proof}

\begin{lemma}\label{L: suma za tip C$||$}For $1\leq r\leq 2n-1$ we have:
$$
\begin{gathered}
\Sigma_T(C_{||\, r})=2\sum_{d=1}^{2n-r}\ \sum_{h=1}^{2n+1-d-r}\ \sum_{\ell=d+h+r}^{2n+1}
\sum_{i_1=r}^{\ell-d-h}\ \sum_{i_2=r-1}^{i_1-1}\ \dots
\ \sum_{i_{r}=1}^{i_{r-1}-1}\ \pi(d,h,\ell; i_1,\dots,i_r),\\
\pi(d,h,\ell; i_1,\dots,i_r)=(4n+1-\ell-d)(\ell-d-h-i_1+1)(i_{1}-i_{2}+1)\cdots(i_{r-1}-i_{r}+1).
\end{gathered}
$$
\end{lemma}

\begin{proof} For $r$ points $a_1, \dots, a_r$ we need $r$ distinct rows, and $\ell$-th row is above them at the distance $d+h$. Hence
$$
d+h+r\leq 2n+1.
$$
For $h=1$ we determine the bounds for $d$, for chosen $d$ we determine the bounds for $h$, and for chosen $h$ we determine the bounds for $\ell$: 
$$
1\leq d\leq 2n-r,\quad 1\leq h\leq 2n+1-r-d, \quad r+d+h\leq \ell\leq 2n+1.
$$

There are $4n+1-\ell-d$ ways to put $c$ on the line $\ell$ at distance $d$ from $b$.
For fixed $d$, $h$ and $c$ on the line $1\leq\ell\leq 2n+1$, the point $a_1$ is on the line $i_1\leq \ell-d-h$. Now we proceed as in the proof of Lemma \ref{L: suma za tip A}, the only difference being that for the vertex $y$ in the triangle $\Delta(b,y,c)$ there is
$$
\ell-d-h-i_1+1
$$
places on the line $i_1$ which are in the cone below $y$.
\end{proof}

\begin{lemma}\label{L: suma za tip D$||$} For $2\leq r+s\leq 2n-1$ we have:
$$
\begin{gathered}
\Sigma_T(D_{r\,||\, s})=2\sum_{d=1}^{ \lfloor \tfrac{2n+1-r-s}{2}\rfloor}\ 
\sum_{h=1}^{2n+2-2d-r-s} \ \sum_{\ell=s+d+h}^{2n+2-r-d} 
\ \sum_{i_1=\ell-r+d+1}^{2n+1}\ \sum_{i_2=\ell-r+d+2}^{i_1-1}\ \dots
\ \sum_{i_{r}=\ell+d}^{i_{r-1}-1}\\
\ \sum_{l_1=s}^{\ell-d-h}\ \sum_{l_2=s-1}^{l_1-1}\ \dots
\ \sum_{l_{s}=1}^{l_{s-1}-1}\ \pi(d,\ell; i_1,\dots,i_r;l_1,\dots,l_s),\\
\pi(d,\ell; i_1,\dots,i_r;l_1,\dots,l_s)=(4n+1-i_1)(i_{1}-i_{2}+1)\dots(i_{r-1}-i_{r}+1)
(i_r-\ell-d+1)\\
(\ell-d-h-l_1+1)(l_{1}-l_{2}+1)\dots(l_{s-1}-l_{s}+1).
\end{gathered}
$$
\end{lemma}

\begin{proof} 
We need to place $r$ distinct rows above $\ell$ at the distance $d$, and $s$ distinct rows below $\ell-h$ at the distance $d$. Hence $r-1+s-1+2d+h\leq 2n+1$. For $h=1$ we determine the bounds for $d$, for chosen $d$ we determine the bounds for $h$, and for chosen $h$ we determine the bounds for $\ell$: 
$$
1\leq d\leq \lfloor \tfrac{2n+2-r-s}{2}\rfloor,\quad 1\leq h\leq 2n+3-2d-r-s,\quad 
s+d+h\leq\ell\leq 2n+1-r-d.
$$
Now we proceed as in the proof of lemmas above: the factors in $\pi(d,\ell; i_1,\dots,i_r;l_1,\dots,l_s)$ count the number of points we can place in one row when a point in the row above is chosen. In particular, the factor $(i_r-\ell-d+1)$ counts the number of places the vertex $x$ of the parallelogram $P(b,y,c,x)$ can take on the line $\ell+d$ in the cone below $a_r$, and $(\ell-d-h-l_1+1)$ is the number of places on the line $l_1$ which are in the cone below the vertex $y$.
\end{proof}

\section{Level $2$ relations for $C_n^{(1)}$ for two embeddings in $\pi$, $\ell(\pi)=4$}

In this section we prove our main result:
\begin{theorem}\label{T:the main theorem}
	For any two embeddings
	$\rho_1 \subset \pi$ and $\rho_2 \subset \pi$ in $\pi\in\mathcal P^{4}(m)$,
	where $\rho_1, \rho_2 \in\ell \!\text{{\it t\,}}(\bar{R})$, we
	have a level $2$ relation for $C_n^{(1)}$
	\begin{equation}\label{E:9.2}
		u(\rho_1 \subset \pi) - u(\rho_2 \subset \pi) 
		=\sum_{\pi\prec \pi', \ \rho\subset\pi'}c_{\rho\subset\pi'}\,u(\rho\subset\pi').
	\end{equation}
\end{theorem}

For a colored partition (\ref{E:colored partition}) we define {\it the shape of $\pi$} as the ordinary partition of  $|\pi|\in-\mathbb N$
$$
\text{sh\,}\pi=\prod\sb{a\in\bar B_{<0}} |a|^{\pi(a)}
$$
with parts $|a|\in-\mathbb N$ for $a\in\text{supp\,}\pi =\{a\in\bar B_{<0}\mid\pi(a)>0\}$. Note that $|\text{sh\,}\pi|=|\pi|=\sum\sb{a\in\bar B_{<0}} |a|\cdot{\pi(a)}$.
We shall depict the ordinary partition $p=(-3)^2(-2)$ with its Young tableaux
$$	
\begin{array}{cl}
&\sq\sq\sq  \\
&\sq\sq\sq  \\
&\sq\sq   \\
\end{array} .   		
$$

To prove the theorem, for a fixed ordinary partition $p$ of length $4$ we need to count ``the number of  two-embeddings for $\text{sh\,} \pi=p$''
$$
\sum_{\text{sh\,} \pi=p, \, \pi\in\mathcal P^4(m)} N(\pi).  
$$
It turns out it is enough, but much easier, to count for a trapezoid $T$ the number
$$
N_T=\sum_{\pi, \, \ell(\pi)=4, \, \text{supp\,} \pi\subset T} N(\pi).  
$$
By Lemma \ref{L: broj potrebnih relacija} we have,
$$
\begin{array}{lll}
	N_T(A_4) = 3\,\Sigma_T(A_4), & N_T(A_3) = 6\,\Sigma_T(A_3),& N_T(A_2) = 3\,\Sigma_T(A_2),\\
	N_T(B_{2\,||})= \Sigma_{T}(B_{2\,||}),& N_T(B_{2\,|})=\Sigma_{T}(B_{2\,|}),&\\
	N_T(B_{1\,||})=\Sigma_{T}(B_{1\,||}),&	N_T(B_{1\,|})=\Sigma_{T}(B_{1\,|}),&\\
	N_T(C_{||\, 2})=\Sigma_{T}(C_{||\, 2}),&	N_T(C_{|\, 2})=\Sigma_{T}(C_{|\, 2}),&\\
	N_T(C_{||\, 1})=\Sigma_{T}(C_{||\, 1}),&	N_T(C_{|\, 12})=\Sigma_{T}(C_{|\, 1}),&\\
	N_T(D_{1 \,||\, 1})= \Sigma_{T}(D_{1 \,||\, 1}),&N_T(D_{1 \,|\, 1})=\Sigma_{T}(D_{1 \,|\, 1}),&
\end{array}
$$ 
so by using the software package \emph{Mathematica} for sumations in  Lemmas \ref{L: suma za tip A}--\ref{L: suma za tip D$||$} we get:
\begin{lemma}\label{L:rezultati prebrajanja}
Let $T$ be a trapezoid consisting of three consecutive triangles in the array ${\bar B}_{<0}$ and   $k=2$. Then 


\begin{align}
	N_T(A_4) & =  \frac{48n^8+672n^7+2296n^6-4613n^4+798n^3+1009n^2-210n}{280},\nonumber\\
	N_T(A_3)& =  \frac{56n^6+420n^5+590n^4-225n^3-151n^2+30n}{15},\nonumber\\	
	N_T(A_2)& =  \frac{20n^4+60n^3+19n^2-3n}{2},\nonumber
	\end{align}
\begin{align}
	N_{T}(B_{2\,|})& =  \frac{128n^7+952n^6+1652n^5+490n^4-553n^3-182n^2+33n}{630},  \nonumber\\
	N_{T}(B_{1\,|})& =  \frac{48n^5+140n^4+120n^3+25n^2-3n}{30},\nonumber\\
	N_{T}(B_{2\,||})& =  \frac{144n^8+992n^7+840n^6-1456n^5-1239n^4+518n^3+255n^2-54n}{1260},\nonumber\\
	N_{T}(B_{1\,||})& =  \frac{56n^6+132n^5+50n^4-45n^3-16n^2+3n}{45}, \nonumber\\
	N_{T}(C_{2\,|})& =  \frac{16n^7+112n^6+160n^5-20n^4-101n^3-2n^2+15n}{90},\nonumber\\
	N_{T}(C_{1\,|})& =  \frac{8n^5+20n^4+10n^3-5n^2-3n}{6}, \nonumber\\
	N_{T}(C_{2\,||})& =  \frac{32n^8+208n^7+112n^6-392n^5-182n^4+217n^3+38n^2-33n}{315}, \nonumber\\
	N_{T}(C_{1\,||})& =  \frac{16n^6+32n^5-20n^3-n^2+3n}{15},\nonumber\\
	N_{T}(D_{1|1}) & =  \frac{256n^7+1512n^6+2884n^5+1575n^4-686n^3-567n^2+66n}{2520},	\nonumber\\
	N_{T}(D_{1||1}) & =  \frac{72n^8+368n^7+448n^6-322n^5-707n^4-28n^3+187n^2-18n}{1260}. \nonumber
\end{align}

\end{lemma}

\begin{lemma}\label{L:N on T}
Let $T$ be a trapezoid consisting of three consecutive triangles in the array ${\bar B}_{<0}$ and   $k=2$. Then
	$$
	N_T=	\sum_{r=2}^{4}N_T(A_r) + \sum_{r=1}^{2}\sum_{\delta=\,|,\, ||}
	\left(N_T(B_{r\,\delta})+N_T(C_{\delta \,r})+N_T(D_{1 \,\delta\, 1})\right)  =\frac{7(10n-1)}{4}{2n+6\choose 7}.
	$$
\end{lemma}
\begin{proof}
 Lemma \ref{L:rezultati prebrajanja} gives
\begin{align*}
N_T &=\frac{1120n^8+11648n^7+47824n^6+9800n^5+103390^4+50372n^3+6426n^2-1260n}{2520}\\
	& =\frac{(10n-1)n(n+1)(n+2)(n+3)(2n+1)(2n+3)(2n+5)}{180} \\
	& =\frac{(10n-1)2n(2n+2)(2n+4)(2n+6)(2n+1)(2n+3)(2n+5)}{180\cdot 16}\\
	& =\frac{7(10n-1)}{4}{2n+6\choose 7}.
\end{align*}
\end{proof}
\begin{proof} [Proof of Theorem \ref{T:the main theorem}]
By Theorem 9.2 and Proposition 10.1 in \cite{PS1} for  $\pi\in\mathcal P^{k+2}(m)$ we have a space of relations for annihilating fields 
\begin{equation}
Q_{k+2}(m)  =  U(\mathfrak g)q_{(k+1)\theta}(m)\oplus U(\mathfrak g)q_{(k+2)\theta}(m)\oplus U(\mathfrak g)q_{(k+2)\theta-\alpha\sp*}(m)\ .
\end{equation}
Especially, for level $k=2$ and each degree $m\leq -k-2$ from the above equation, it follows
\begin{equation}
Q_4(m)=U(\mathfrak g)q_{3\theta}(m)\oplus U(\mathfrak g)q_{4\theta}(m)\oplus U(\mathfrak g)q_{4\theta-\alpha\sp*}(m)\ 
\end{equation}
where $\alpha\sp*=\alpha_1=\varepsilon_1 -\varepsilon_2$.
Moreover, again thanks to the mentioned Theorem 9.2, the proof of Theorem \ref{T:the main theorem} is reduced to the verification of equality
\begin{equation}\label{E:9.1}
	\sum_{\pi\in\mathcal P^4(m)} N(\pi)=\dim Q_4(m)\ .
\end{equation}
The Weyl dimension formula 
in the case of symplectic Lie algebra $\mathfrak{g}=\mathfrak{sp}_{2n}$\\ (with the corresponding $\rho=n\varepsilon_1+(n-1)\varepsilon_2+\cdots +2\varepsilon_{n-1}+\varepsilon_n$) gives
\begin{eqnarray}
	\dim L(s\theta) &=& {2n+2s-1\choose 2s},\label{E:11.1}\\
	\dim L(4\theta-\alpha^{\star}) &=& \frac{(2n+7)(n-1)}{4}{2n+5\choose 6}.\label{E:11.2}
\end{eqnarray}
Hence from (\ref{E:11.1}) and (\ref{E:11.2}) we have
\begin{equation}\label{tri sume}
	\dim Q_4(m) = \dim L(3\theta) + \dim L(4\theta)  + \dim L(4\theta-\alpha^{\star}) =2n
	{2n+6\choose 7} .
\end{equation}
Since all technical Lemmas \ref{L: klasifikacija dva ulaganja} - \ref{L: suma za tip D$||$} were proved for three successive triangles (i.e. for the trapezoid)  the equality (\ref{E:9.1}) can be replaced with the equivalent one
\begin{equation}\label{sum P^4(n)}
	\sum_{m=4}^{12}\sum_{\pi\in\mathcal P^4(m)} N(\pi)  = 9\times 2n {2n+6\choose 7}- 2\times\dim L(4\theta)= \frac{7(10n-1)}{4}{2n+6\choose 7}\ .
\end{equation}
Indeed, using Young tableaux we can list the following cases
$$	\begin{array}{clllll}
		m=-4&m=-5&m=-6&m=-7&m=-8&m=-8\\
		\sq &\sq\sq&\sq\sq&\sq\sq&\sq\sq&\sq\sq\sq \\
		\sq &\sq &\sq\sq&\sq\sq&\sq\sq&\sq\sq \\
		\sq &\sq &\sq&\sq\sq&\sq\sq&\sq\sq \\
		\sq &\sq &\sq&\sq &\sq\sq&\sq \\       
	\end{array}    		
$$
$$	\begin{array}{lllllc}
	m=-9&m=-10&m=-11&m=-12&m=-13&\cdots\\
	\sq\sq\sq&\sq\sq\sq&\sq\sq\sq&\sq\sq\sq&\sq\sq\sq\sq&\cdots \\
	\sq\sq&\sq\sq\sq&\sq\sq\sq&\sq\sq\sq&\sq\sq\sq&\cdots \\
	\sq\sq&\sq\sq&\sq\sq\sq&\sq\sq\sq&\sq\sq\sq&\cdots \\
	\sq\sq&\sq\sq&\sq\sq&\sq\sq\sq&\sq\sq\sq&\cdots       
\end{array}    	
$$
\begin{center}
\text{Figure 7}
\end{center}

It is important to note that Young tableauxes from $m=-4$ to $m=-12$ covers all cases which can occur in the trapezoid scheme of three successive triangles. Now, from Lemma \ref{L:N on T} and  (\ref{sum P^4(n)}) it is obvious that equality (\ref{E:9.1}) holds.
\end{proof}

 \begin{remark}
From Lemmas \ref{L: suma za tip A}--\ref{L: suma za tip D$||$} and  Faulhaber's formula (cf. \cite{O}) we see that $N_T$ is a polynomial in $n$ of degree at most $8$. So it is enough to check Lemma \ref{L:N on T} for $9$ values of $n$, say for $n=1, \dots, 9$. The value of $N_T$ for $n=1$ equals $126$ and this corresponds to the ``second simplest case'' in \cite{MP1}. The value of $N_T$ for $n=9$ equals $53\,905\,698$---this number illustrates the degree of complexity when passing from $C_1^{(1)}=A_1^{(1)}$ to  $C_n^{(1)}$ for higher ranks $n$ (of course, this discussion goes only for level $k=2$). 
\end{remark} 

\section*{Acknowledgement}

This work is partially supported by the QuantiXLie Centre of Excellence, a project cofinanced by the Croatian Government and European Union through the European Regional Development Fund---the Competitiveness and Cohesion Operational Programme (Grant KK.01.1.1.01.0004).


\end{document}